\documentclass{amsart}
\usepackage{amssymb}

   \theoremstyle{plain}
\newtheorem{theorem}{Theorem}[section]
\newtheorem{proposition}[theorem]{Proposition}

\newtheorem{corollary}[theorem]{Corollary}

	 \theoremstyle{definition}
\newtheorem{definition}[theorem]{Definition}

	 \theoremstyle{remark}
\newtheorem{remark}{Remark}

\newcommand\lle{\mathbin{*{\le}}}
\newcommand\rle{\mathbin{{\le}*}}
\newcommand\BH{\mathcal B(H)}
\newcommand\cle\preceq
\newcommand\cwedge\curlywedge
\newcommand\cvee\curlyvee
\newcommand\lcle{\mathbin{*{\cle}}}
\newcommand\rcle{\mathbin{{\cle}*}}
\newcommand\dg{^\dagger}
\newcommand\lt{^\backprime}
\newcommand\ltt{^{\backprime\backprime}}
\newcommand\rcvee{\mathop{{\cvee}\!*}}
\newcommand\rcwedge{\mathop{{\cwedge}\!*}}

\numberwithin{equation}{section}
\newcommand\colo{\colon\,}
\newcommand\ran{\mathop{\mathrm{ran}}}
\newcommand\rran[1]{\overline{\ran #1}}
\newcommand{\stminus}
{\mathrel{^*}\hspace{-.9em}-} 
\newcommand\thitem[1]{\item[\hspace*{-3ex}{\upshape (#1)}]}
\newcommand\pritem[1]{\item[\hspace*{5ex}(#1)]}

\title[On one-sided star partial orders on a Rickart *-ring]{On one-sided star partial orders \\ on a Rickart *-ring}
\author[J.\, C\={\i}rulis]{J\=anis C{\=\i}rulis}
\address{Institute of Mathematics and Computer Science \\ University of Latvia}
\email{jc@lanet.lv}

\subjclass{06A06; 16W10; 47A05}
\keywords{bounded linear operator; left-star order; right-star order; Rickart *-ring, *-regular ring}

\begin{document}

\begin{abstract}
We compare some recent approaches to extending the notions of left- and right-star partial order, introduced for complex matrices in early 90-ies, to bounded linear Hilbert space operators and to certain *-rings, and discuss in more detail a version of right-star partial order on a  Rickart *-ring $R$. Regularity of $R$ is not presumed.
\end{abstract}

\maketitle

\section{Introduction}	\label{intro}

One-sided (i.e., left- and right-) star orders for $m \times n$ matrices was introduced by Baksalary and Mitra in \cite{BM}  (see also \cite{MBM}) and have been intensively studied. The corresponding definitions are

\begin{eqnarray}
A \lle B &\text{iff}& A^*A = A^*B \text{ and Im}A \subseteq \text{Im}B, \label{M1}     \\
A \rle B &\text{iff}& AA^* = BA^* \text{ and Im}A^* \subseteq \text{Im}B^*. \label{M2}
\end{eqnarray}
Both orders have also been transferred to bounded linear Hilbert space operators. For example, the definitions assumed in \cite{DW} are direct analogues of those for the matrix case (to unify notations we borrow here, and in the sequel, those used in 
\cite{C}, Example 1: $\ran A$, $\rran A$ and $\ker A$ stand for the range, the closed range and the nullspace of an operator $A$, respectively):
\begin{eqnarray}
A \lle B &\text{iff}& A^*A = A^*B \text{ and } \ran A \subseteq \ran B, \label{D1}     \\
A \rle B &\text{iff}& AA^* = BA^* \text{ and } \ran A^* \subseteq \ran B^*. \label{D2}
\end{eqnarray}
These orders are not independent: evidently, $A \lle B$ if and only if $A^* \rle B^*$. It is observed in Theorem 2.2 of \cite{DW} that
$A \lle B$ if and only if there are invertible operators $E,F$ such that $EAF \rle EBF$. It is known well that range inclusion of operators in a Hilbert space can be characterized algebraically:
\[
\ran A \subseteq \ran B \text{ iff there is an operator } C \text{ such that } A = BC
\]
(see, e.g.\ \cite[Lemma 2.1]{DW}). Therefore, the definitions (\ref{D1}) and (\ref{D2}) can be given a form
\begin{eqnarray}
A \lle B &\text{iff}& A^*A = A^*B \text{ and } A = BC \text{ for some } C, \label{D1'}     \\
A \rle B &\text{iff}& AA^* = BA^* \text{ and } A = CB \text{ for some } C \label{D2'}
\end{eqnarray}
suitable for immediate transferring  to rings with involution.

In \cite{DGM}, the left-star order for operators is defined as follows
\begin{eqnarray}
A \lle B & \text{ iff }& \ran P = \rran{A}, \ker A = \ker Q, PA = PB, AQ = BQ \label{dmg1} \\[-1ex]
& & \mbox{for some appropriate projection operator $P$ and} \nonumber \\[-1ex]
& & \mbox{an idempotent operator $Q$} \nonumber;
\end{eqnarray}
it is then proved that the defined relation is a  partial order indeed and that this definition is equivalent to (\ref{D1}). The right-star order $\rle$ is introduced there similarly, by the same condition (\ref{dmg1}) with $P$ idempotent and $Q$ a projection. See also \cite{MRD}.

We note that still another extension of (\ref{M1}) and (\ref{M2}) to $\BH$, the set of bounded linear operators over a Hilbert space $H$, is possible, and introduce two other order relations, $\rcle$ and $\lcle$:
\begin{eqnarray}
A \lcle B &\text{iff}& A^*A = A^*B \text{ and } \rran A \subseteq \rran B, \label{C1}     \\
A \rcle B &\text{iff}& AA^* = BA^* \text{ and } \rran{ A^*} \subseteq \rran{ B^*}. \label{C2}
\end{eqnarray}
Generally, they are weaker than $\lle$ and $\rle$; however, the difference disappears if the underlaying Hilbert space is finite dimensional. In the infinite-dimensional case, an operator in  $\BH$ has a closed range if and only if it is regular (has the Moore-Penrose inverse); therefore, both versions of one-sided star orders coincide on regular operators.

From this point of view, both (\ref{D1}) and (\ref{C1}) are equally appropriate generalizations of the matrix ordering (\ref{M1}), and the same concerns also (\ref{D2}), (\ref{C2}) and (\ref{M2}). Notice that the defining conditions of $\lcle$ and $\rcle$ also can be rewritten purely in terms of operators, as the lattice of closed subspaces of $H$ is isomorphic to that of projection operators. Given an operator $C$, let as denote by $P_C$ the projection operator onto the closure of $\ran C$. Then
\begin{eqnarray}
A \lcle B &\text{iff}& A^*A = A^*B \text{ and } P_A \le P_B, \label{C1'}     \\
A \rcle B &\text{iff}& AA^* = BA^* \text{ and } P_{ A^*} \le P_{ B^*}, \label{C2'}
\end{eqnarray}
where $\le$ stands for the natural ordering of projection operators. This form of definitions allows us to transfer them naturally to Rickart *-rings. This transfer is the main purpose of the present paper.

The next section contains the necessary preliminary information, mainly from \cite{C}, on *-rings with certain additional structure. In Section \ref{ord}, we compare a few algebraic analogues, found in  the literature, of the definitions discussed in Introduction. For illustration, some elementary properties of the right-star order $\rcle$ in a Rickart *-ring are stated in the final section \ref{latt}.

\section{Preliminaries: regular and Rickart *-rings}   \label{prelim}

We shall deal only with associative rings. Recall that a ring is said to be \emph{regular} if every its principal right (equivalently, left) ideal is generated by an idempotent. A \emph{*-regular ring} is a regular *-ring in which the involution is proper, i.e., $x = 0$ in it whenever $x^*x = 0$ (or, equivalently, whenever $xx^* = 0$). An involution  ring is *-regular if and only if every element $x$ in it has a Moore-Penrose inverse $\dg$ (necessary unique), which is characterized by four identities $xx\dg x = x$, $x\dg xx\dg = x\dg$, $(xx\dg)^* = xx\dg$, $(x\dg x)^* = x\dg x$.

Self-adjoint idempotents of an involution ring are called \emph{projections}. A \emph{Rickart *-ring} may be defined as a ring with involution in which the left and right annihilators of every element are principal ideals each generated by a projection (see \cite{C,B,J2}); in fact, any one of the two conditions suffices. Put in another way, this means that a *-ring is Rickart if, given $x$, we can choose projections $x\lt$, resp., $x'$ such that, for all elements $y,z$ of the ring,
\begin{equation}  \label{primes1}
yx = 0 \ \text{ iff } \ yx\lt = y,
\ \text{ resp., }\
xz = 0 \ \text{ iff } \ x'z = z .
\end{equation}
The projections are actually unique for every $x$; hence
\begin{equation}   \label{primes2}
x\ltt = 1 - x\lt, \quad x'' = 1 - x' .
\end{equation}
These identities imply that the conditions in (\ref{primes1}) are equivalent to
\begin{equation}	\label{primes3}
yx = 0 \ \text{ iff } \ yx\ltt = 0, \quad
xz = 0 \ \text{ iff } \ x''z = 0
\end{equation}
respectively. In fact, $x\lt = (x^*)'$, $x' = (x^*)\lt$ and $x\ltt = x\lt{}'$, $x'' = x'{}\lt$. A Rickart ring is unital with   $0\lt = 1 = 0'$.

For example, it follows immediately from the definitions that every *-regular unital ring is a Rickart *-ring. We may put
\[
\mbox{$x\lt := 1 - xx\dg$ and $x' := 1 - x\dg x$}:
\]
then both $x\lt$ and $x'$ are projections and the identities (\ref{primes1}) hold. Furthermore, as $e\dg = e$ for all projections $e$, we obtain (\ref{primes2}); it then follows that
\[
\mbox{$x\ltt = xx\dg$ and $x'' = x\dg x$}.
\]
Conversely, a regular Rickart *-ring is *-regular.

Recall that the idempotent elements of any unital ring form an orthomodular poset with respect their natural order given by $e \le f$ iff $ef = e = fe$. The set $P$ of projections of a Rickart *-ring is even an orthomodular lattice. Moreover, if $e$ and $f$ are projections, then, $e \le f$ iff $ef = e$ iff $fe = e$, and  $ef = 0$ iff $fe = 0$. The orthocomplement of $e$ in $P$ is given by $1 - e$ and is therefore represented by both $e\lt$ and $e'$. It follows that $e \le f$ iff $ef\lt = 0$ iff $fe' = 0$.

\begin{remark} \label{hilb}
It is known well that every ring $B(H)$ may be regarded as a particular Rickart *-ring. To see this, one has to interpret in an appropriate way not only involution (as usual, $A^*$ is the adjoint of the operator $A$) and both prime operations (see below), but also multiplication. In the ring product $ab$, the left multiplier $a$ is normally considered as the first, and the right one, $b$, as the second multiplier. To treat an operator product $AB$ in the same way, one has to regard operators from $B(H)$ as operating in $H$ on the right (as in Example 1 of \cite{C})---otherwise the product $AB$ has to be represented in an abstract ring as $ba$. Neither action is the usual practice; so, to be consistent with the other papers discussed in Introduction and Section \ref{ord}, we assume here that a ring product $ab$ is transliterated into the operator notation as $AB$, and conversely.

Then, for $A \in B(H)$, $A\lt$ is the projection onto the orthogonal complement of $\ran A$ (i.e., onto $\ker A^*$, $A'$ is the projection onto $\ker A$, $A\ltt$ is the projection $P_A$, and $A''$ is the projection onto the orthogonal complement of $\ker A$ (i.e., the projection $P_{A^*}$).
\end{remark}

We now list a number of elementary properties of Rickart *-rings. According to \cite[Lemma 3.2]{C} or \cite[Lemmas 2.4, 2.5]{MRD},
\begin{gather}
a^*a = a^*b \text{ iff } a = a\ltt b \text{ iff } a = pb \text{ for some projection } p, \label{lstar1} \\
aa^* = ba^* \text{ iff } a = ba'' \text{ iff } a = bq \text{ for some projection } q.  \label{rstar1}
\end{gather}
in any Rickart *-ring.
The right-hand relationships stated in the subsequent proposition were obtained in \cite[Proposition 2.4]{C}. The proofs given below for (a)--(d) are more straightforward.

\begin{proposition} \label{doubles}
In a Rickart *-ring
\begin{enumerate}
\thitem{a} $a\lt a = 0$, \quad $aa' = 0$,
\thitem{b} $a\ltt a = a$, \quad $aa'' = a$,
\thitem{c} $(ab)\ltt \le a\ltt$, \quad $(ab)'' \le b''$,
\thitem{d} $(ab)\ltt = (ab\ltt)\ltt$, \quad $(ab)'' = (a''b)''$,
\thitem{e} if $e \le a\ltt$, then $(ea)\ltt = e$, \quad if $e \le a''$, then $(ae)'' = e$,
\thitem{f} the subsets $\{e\colo ea = 0\}$ and $\{e\colo ae = 0\}$ are sublattices of $P$ for every $a$.
\end{enumerate}
\end{proposition}
\begin{proof}
We shall consider only the ``right'' case.
\begin{enumerate}
\thitem{a} By (\ref{primes1}), as $a'$ is idempotent.
\thitem{b} By (\ref{primes2}) and (a), $aa'' = a(1 - a') = a$.
\thitem{c} $(ab)'' \le b''$ iff $(ab)''b' = 0$ iff $abb' = 0$ (see (\ref{primes3}) and (a)).
\thitem{d} Likewise $(ab)'' \le (a''b)''$ iff $(ab)''(a''b)' = 0$ iff $ab(a''b)' = 0$ iff $a''b(a''b)'= 0$ and $(a''b)'' \le (ab)''$ iff $(a''b)''(ab)' = 0$ iff $a''b(ab)'= 0$ iff $ab(ab)' = 0$.
\thitem{e} If $a''e = e$, then, by (d), $(ae)'' = (a''e)'' = e'' = e$.
\thitem{f} For $ae = 0$ iff $a'e = e$ iff $e \le a'$; of course, every initial segment of $P$ is a sublattice.
\end{enumerate}
\end{proof}

\section{One-sided star orders in rings}
\label{ord}

In \cite{LPT}, the following pair of definitions of left and right star partial order in a *-regular ring $R$ is used:
\begin{eqnarray}
a \lle b &\text{iff}& a^*a = a^*b \text{ and } aR \subseteq bR, \label{L1}     \\
a \rle b &\text{iff}& aa^* = ba^* \text{ and } Ra \subseteq Rb. \label{L2}
\end{eqnarray}
Standard calculations show that
\begin{equation}    \label{RR}
 aR \subseteq bR \text{ iff } a = bx \text{ for some } x, \quad Ra \subseteq Rb \text{ iff } a = yb \text{ for some } y ;
\end{equation}
therefore (\ref{L1}) and (\ref{L2}) are in fact the abstract analogues of (\ref{D1'}) and (\ref{D2'}).

The definitions (\ref{L1}) and (\ref{L2}) can be rewritten in terms of Rickart prime operations. First, due to (\ref{lstar1}) and (\ref{rstar1}), involution can be eliminated. Further, as the ring $R$ above is supposed to be *-regular, the two equivalences in (\ref{RR}) my be specified as follows:
\begin{equation*}
 aR \subseteq bR \text{ iff } bb\dg a = a, \quad Ra \subseteq Rb \text{ iff } ab\dg b = a .
\end{equation*}
Equivalently,
\begin{equation}    \label{RR'}
 aR \subseteq bR \text{ iff } b\ltt a = a, \quad Ra \subseteq Rb \text{ iff } ab'' = a .
\end{equation}
We thus come to the following conclusion concerning the one-sided star partial orders introduced by (\ref{L1}) and (\ref{L2}).

\begin{theorem} \label{rrr}
In a regular Rickart *-ring,
\[
\mbox{$a \lle b$ \;iff\; $a\ltt b = a = b\ltt a$,
\quad
$a \rle b$ \;iff\; $ba'' = a = ab''$}.
\]
\end{theorem}

The second identity in each equivalence can be further modified using the following easy consequences of (\ref{primes1})--(\ref{primes3}):
\begin{gather}
b\ltt a = a \;\text{ iff }\; b\lt a = 0 \;\text{ iff }\; b\lt a\ltt = 0 \;\text{ iff }\; b\ltt a\ltt = a\ltt \;\text{ iff }\; a\ltt \le b\ltt, \label{lstar2}
\\
ab'' = a \;\text{ iff }\; ab' = 0 \;\text{ iff }\; a''b' = 0 \;\text{ iff }\; a''b'' = a''  \;\text{ iff }\; a'' \le b''. \label{rstar2}
\end{gather}

The authors of \cite{MRD} provide a definition of a left-star partial order in a unitary involution ring which is an abstract algebraic analogue of (\ref{dmg1}), and then state in Theorem 9 that it is equivalent, in a Rickart *-ring, to
\begin{eqnarray}
a \lle b &\text{ iff }& 
    a^*a = a^*b \text{ and } a = bq \text{ for some idempotent } q \label{mrd1} \\
& & \mbox{such that, for every $x$, $qx = 0$ iff $ax = 0$} \nonumber.
\end{eqnarray}
The right-star partial order $\rle$ is introduced and analogously. They also prove that $\lle$ and $\rle$ are indeed partial orders in a Rickart *-ring (Theorem 10 in \cite{MRD}). Further, in a regular Rickart *-ring (i.e., a *-regular ring), these definitions of $\lle$ and $\rle$ are shown in \cite[Theorem 14]{MRD} to be equivalent to the following ones (in the notation of the present paper):
\begin{eqnarray}
a \lle b &\text{ iff }& a^*a = a^*b \text{ and } b\lt a = 0,  \label{M1'} \\
a \rle b &\text{ iff }& aa^* = ba^* \text{ and } ab'= 0 . \label{M2'}
\end{eqnarray}
By (\ref{lstar1}), (\ref{rstar1}) and (\ref{lstar2}),
(\ref{rstar2}), it is now evident that (\ref{M1'}), (\ref{M2'}) are equivalent to the pair of characteristics of the relations $\lle$, $\rle$ given in Theorem \ref{rrr}. (It follows that the relations $\lle$ and $\rle$ in Theorem \ref{rrr} and, hence, in (\ref{L1}), (\ref{L2}) also are partial orders; this was not explicitly stated in \cite{LPT}).

The characteristics of one-sided orders given in Theorem \ref{rrr} make sense in an arbitrary  Rickart *-ring; however, they need not hold true there. The reason is that the equivalences (\ref{RR}) generally (i.e., without regularity) do not follow from (\ref{RR'}). This observation motivates introducing of a new pair of relations.
\begin{definition} \label{CC}
The left (right) star partial orders $\lcle$ and $\rcle$ on a Rickart *-ring are defined by
\[
a \lcle b \text{ iff } a\ltt b = a = b\ltt a,
\quad
a \rcle b \text{ iff } ba'' = a = ab''.
\]
\end{definition}

Observe that, due to (\ref{lstar1}), (\ref{rstar1}), (\ref{lstar2}), (\ref{rstar2}) and in virtue of the last paragraph of the Remark \ref{hilb},
\[
\mbox{
$a \lcle b$ iff $a^*a = a^*b$ and $a\ltt \le b\ltt$,
\quad
$a \rcle b$ iff $aa^* = ba^*$ and $a'' \le b''$};
\]
these equivalences are the abstract analogues of (\ref{C1'}) and (\ref{C2'}). They go back to \cite[Remark 2]{C}; however, (i) as it was assumed in \cite{C} that operators in a Hilbert space operate on the right, the left and right order were interchanged there to fit with this assumption, (ii) by an elementary fault in reasoning, both equivalences were presented there as algebraizations of (\ref{D1}), (\ref{D2}) rather than (\ref{C1}), (\ref{C2}).

\begin{theorem}
The relations $\lcle$ and $\rcle$ are partial orders.
\end{theorem}
\begin{proof}
We shall consider only the right star order, and shall use (\ref{primes1}) and (\ref{rstar1}). Evidently, the relation $\rcle$ is reflexive. It is transitive: if $a \rcle b$ and $b \rcle c$, then $ca'' = cb''a'' = ba'' = a$ and $ac'' = a b''c'' = ab'' = a$. It is also antisymmetric: if $a \rcle b$ and $b \rcle a$, then $b = ab'' = ba''b'' = ba'' = a$.
\end{proof}

\begin{remark}
Definition \ref{CC} makes sense even in an ordinary Rickart ring; however, some additional (involution-free) assumptions are necessary to assure equivalences (\ref{lstar1}) and (\ref{rstar1}), which are needed for the above theorem. We shall discuss this generalization in another paper.
\end{remark}

\section{$\rcle$-order structure of Rickart *-rings}	
\label{latt}

Let $R$ be some Rickart *-ring. In this section, we shall deal only with the right star partial order $\rcle$ on it.

The subsequent proposition is an analogue of Lemma 3.1 in \cite{C} for the star order on a Rickart *-ring. 

\begin{proposition} \label{prop1}
In $R$,
\begin{enumerate}
\thitem{a}
$0$ is the least element,
\thitem{b}
the right star partial order coincides on $P$ with the usual order of idempotents,
\thitem{c}
$a \in P$ if and only if $a \rcle 1$,
\thitem{d}
every left invertible element is maximal.
\end{enumerate}
\end{proposition}
\begin{proof}
(a) evident.
\pritem{b} For $e,f \in P$, $e \rcle f$ iff $fe = e = ef$ iff $e \le f$.
\pritem{c} $a \rcle 1$ iff $a'' = a$. Now notice that the right-hand equality holds if and only if $a$ is a projection.
\pritem{d} If $ya = 1$, then $1 = (ya)'' \le a''$ (Proposition \ref{doubles}(c)) and $a'' = 1$, Now, if $a \rcle z$, then  $z = a$.
\end{proof}

The following theorem is the central result in this section.
\begin{theorem}
Let $x$ be any element of $R$. The mappings $\phi\colo [0,x] \to P$ and $\psi_x\colo [0,x''] \to R$, defined by
\[
\phi(a) := a'', \quad \psi_x(e) := xe,
\]
are mutually inverse and establish an order isomorphism between the initial segments $[0,x] \subseteq R$ and $[0,x''] \subseteq P$.
\end{theorem}
\begin{proof}
If $a \rcle b \rcle x$, then, evidently, $\phi(a) \le \phi(b) \le x''$ (see (\ref{rstar1})) and $\psi_x(\phi(a)) = a$ (Definition \ref{CC}). In particular, $\phi$ is isotone. Similarly, if $e \le f \le x''$, then $\psi_x(e) \rcle \psi_x(f) \rcle x$. Indeed, by Proposition \ref{doubles}(e), $\psi_x(f)((\psi_x(e))'' = xf(xe)'' = xfe = xe = \psi_x(e)$  and $\psi_x(e)(\psi_x(f))'' = xef = xe = \psi_x(e)$; thus, $\psi_x$ is isotone, and then $\psi_x(f) \rcle \psi_x(x'') = xx'' = x$ (Proposition \ref{doubles}(b)). Further, $\phi(\psi_x(e)) = e$  (Proposition \ref{doubles}(e)); together with the similar identity stated at the beginning of the proof, this implies that $\phi$ and $\psi$ are mutually inverse. In addition, the range of $\phi$ is $[0,x'']$, and the range of $\psi_x$ is $[0,x]$.
\end{proof}

\begin{corollary}
Every initial segment $[0,x]$ with $x$ maximal is order isomorphic to $P$.
\end{corollary}

\begin{corollary}
Initial segments $[0,x]$ and $[0,y]$ of $R$ are order isomorphic if and only if so are the segments $[0,x'']$ and $[0,y'']$ of $P$.
\end{corollary}

We can say more on the order structure of initial segments.

\begin{corollary}
Every interval $[0,x]$ is an orthomodular lattice. Moreover, if $a,b \in [0,x]$, then
\begin{enumerate}
\thitem{a} $x(a'' \wedge b'')$ is the meet of $a$ and $b$ in $[0,x]$,
\thitem{b} $x(a'' \vee b'')$ is the join of $a$ and $b$ in $[0,x]$,
\thitem{c}
$x - a$ is the orthocomplementation of $a$ in $[0,x]$.
\end{enumerate}
\end{corollary}
\begin{proof}
Recall that $P$ is an orthomodular lattice. Hence, every segment $[0,e]$ in $P$ also is a lattice, which also is orthomodular with the orthocomplementation of $f$ given by $f^\perp_e = e - f$. Therefore, any interval $[0,x]$, being an order-isomorphic copy of $[0,x'']$, is an orthomodular lattice with  the join of $a$ and $b$ given by $\psi_x(\phi(a) \vee \phi(b))$, the meet of $a$ and $b$ given by $\psi_x(\phi(a) \wedge \phi(b))$, and the orthocomplementation of $a$ given by $\psi_x(x'' - \phi(a))$ (see Proposition \ref{doubles}(b) and Definition \ref{CC}).
\end{proof}

Let $\rcwedge$ and $\rcvee$ stand for the join and meet operation respectively in $R$ (with respect to $\rcle$). It easily follows from items (c) and (d) of Proposition \ref{prop1} that, generally, at least the operation $\rcvee$ is partial, for $1$ is an example of a left invertible element of $R$. We can adjust to the present context the Theorem 4.2 of \cite{C} which describes the meet and the join of a bounded pair of elements in a Rickart *-ring under the star order.

\begin{theorem}
If $a,b \rcle x$, then
\begin{enumerate}
\thitem{a}  $a \rcwedge b$ exists and equals to $x(a'' \wedge b'')$,
\thitem{b} $a \rcvee b$ exists and equals to $x(a'' \vee b'')$.
\end{enumerate}
\end{theorem}
\begin{proof}
A straightforward checking shows that the meet of two elements in an initial segment of a poset is also their meet in the whole poset; therefore, (a) immediately follows from the previous theorem. This may be not the case with joins: generally, the join of two elements in an initial segment may be not the least upper bound of them in the poset. Therefore, we have to examine this point in $R$ separately.

Suppose that $a,b \rcle x$ and, consequently, $a = xa''$, $b = xb''$ and $a'',b'' \le x''$. Let $c :=x(a'' \vee b'')$. Then $c$, being the join of $a$ and $b$ in $[0,x]$, is an upper bound of $a$ and $b$ in $R$. Suppose that $y$ is one more such an upper bound; then $a = ya''$, $b = yb''$ and $a'',b'' \le y''$. Hence, $d := yc'' = y(x(a'' \vee b''))'' = y(x''(a'' \vee b''))'' = y(a'' \vee b'')$; by the previous corollary, $d \in [0,y]$. Now, $c = d$: as $(x - y)a'' = 0  = (x - y)b''$, Proposition \ref{doubles}(g) implies that  $(x - y)(a'' \vee b'') = 0$. Thus, $c \rcle y$, and, $c$ is the least upper bound of $a$ and $b$, as desired in (b).
\end{proof}

\begin{corollary}
If $a$ and $b$ have an upper bound, then
$a(a'' \wedge b'') = a \rcwedge b = b(a'' \cwedge b'')$.
\end{corollary}
\begin{proof}
Assume that $a,b \rcle x$. Then, for example, $a(a'' \rcwedge b'') = xa''(a'' \rcwedge b'') = x(a'' \rcwedge b'') = a \rcwedge b$.
\end{proof}

\end{document}